\documentclass[11pt]{amsproc}

\usepackage{amsmath, amssymb, amsthm, latexsym,nicefrac,setspace}
\usepackage[pagebackref,colorlinks,linkcolor=red,citecolor=blue,urlcolor=blue,hypertexnames=true]{hyperref}
\usepackage[backrefs]{amsrefs}

\theoremstyle{plain}
\newtheorem{theorem}{Theorem}[section]
\newtheorem{question}[theorem]{Question}
\newtheorem{lemma}[theorem]{Lemma}
\newtheorem{corollary}[theorem]{Corollary}
\newtheorem{proposition}[theorem]{Proposition}

\theoremstyle{remark}

\newtheorem*{theorem*}{Theorem}

\input{xy}
\xyoption{all}

\def\R{\mathbb{R}}

\def\N{\mathbb{N}}

\title[Length of shortest non-trivial elements]{On the length of the shortest non-trivial element \\ in the derived and the lower central series}

\author{Abdelrhman Elkasapy}
\address{Abdelrhman Elkasapy, MPI-MIS, Inselstra\ss e 22,
04103 Leipzig, Germany, and Mathematics Department, South Valley University, Qena, Egypt}
\email{elkasapy@mis.mpg.de}

\author{Andreas Thom}
\address{Andreas Thom, Univ.\ Leipzig,
PF 100920, 04009 Leipzig , Germany}
\email{andreas.thom@math.uni-leipzig.de}

\begin{document}

\onehalfspace

\begin{abstract}
We provide upper and lower bounds on the length of the shortest non-trivial element in the derived series and lower central series in the free group on two generators. The techniques are used to provide new estimates on the nilpotent residual finiteness growth and on almost laws for compact groups.
\end{abstract}

\maketitle

\tableofcontents

\section{Introduction}
It is a well-known and remarkable theorem of Friedrich Levi \cite{levi1,levi2} that any nested series of subgroups which are characteristic in each other in a free group either stabilizes or has trivial intersection. This is non-trivial to prove directly even for the derived series (see Section \ref{girth} for definitions). Using his non-commutative differential calculus, Ralph Fox \cite{fox} has extended this result to the lower central series and given a conceptual explanation  -- for the lower central series he proved that the length of the shortest non-trivial element in the $n$-th step of this series has length at least $n/2$. It is an interesting question to determine the precise asymptotics of this quantity.
Equivalently, one could ask for some information on the smallest integer $m$, such that every element of length $n$ in the free group survives in some quotient which is $m$-step solvable resp.\ $m$-step nilpotent. Hence, we are trying to make the fact that the free group is residually solvable and residually nilpotent quantitative. Similar questions have been asked in the context of residual finiteness, see \cite{MR2851069, MR2583614, MR2784792, MR2970452} for some recent work on this problem.

In this note we want to provide upper bounds for the growth rate of the length of the shortest non-trivial element in the derived series and the lower central series. An upper bound of $n^2$ was proved by Malestein-Putman \cite{MR2737679} and conjectured to be asymptotically sharp. We disprove this conjecture with a concrete construction.

\vspace{0.2cm}

For a group $\Gamma$ and $a,b \in \Gamma$, we write $[a,b]=aba^{-1}b^{-1}$. We note the basic identities $[a,b]^{-1} = [b,a]$, $[a,a]=[a,a^{-1}] = [a,e]=e$, for all $a,b \in \Gamma$. If $\Lambda_1,\Lambda_2 \subset \Gamma$ are subgroups, we write $[\Lambda_1,\Lambda_2]$ for the subgroup generated by $\{[\lambda_1,\lambda_2] \mid \lambda_1 \in \Lambda_1, \lambda_2 \in \Lambda_2 \}$.

\vspace{0.2cm}

For functions $f,g \colon \N \to \R$, we write
$f(n)=O(g(n))$ if $$\limsup_{n \to \infty} \frac{|f(n)|}{|g(n)|}< \infty.$$ We write $f(n) = o(g(n))$ if 
$$\lim_{n \to \infty} \frac{|f(n)|}{|g(n)|}=0$$
and $f(n) \preceq g(n)$ if there is a constant $C$, such that
$f(n) \leq C g(Cn)$ for all $n \in \N$.

\section{Growth of girth in the lower central and derived series}
\label{girth}
Let ${\mathbb F}_2$ be the free group on two generators $a$ and $b$. We denote the word length function with respect to the generating set $\{a,a^{-1},b,b^{-1}\}$ by $\ell \colon {\mathbb F}_2 \to \N$. Recall that the lower central series is a nested family of normal subgroups of a group $\Gamma$ which is defined recursively by
$$\gamma_1({\Gamma}) := \Gamma \quad \mbox{and} \quad \gamma_{n+1}(\Gamma):= [\gamma_{n}(\Gamma),\Gamma] \quad n \geq 1.$$ We also consider the derived series, which is defined by the recursion 
$$\Gamma^{(0)}:= \Gamma \quad \mbox{and} \quad \Gamma^{(n+1)} := [\Gamma^{(n)}, \Gamma^{(n)}], \quad n \geq 0.$$ It is a well-known fact that $[\gamma_{n}(\Gamma),\gamma_{m}(\Gamma)] \subset \gamma_{n+m}(\Gamma)$, and hence induction can be used to show the inclusions
\begin{equation} \label{relation}
\gamma_n(\gamma_m(\Gamma)) \subset \gamma_{nm}(\Gamma) \quad \mbox{and} \quad \Gamma^{(n)} \subset \gamma_{2^{n}}(\Gamma), \quad \forall n,m \in \N.\end{equation}
Moreover, it is clear from the definition that 
\begin{equation} \label{clear}
(\Gamma^{(n)})^{(m)} = \Gamma^{(n+m)}.
\end{equation}
In this section we want to study the growth of the functions
$$\alpha(n):= \min \{ \ell(w) \mid w \in \gamma_n({\mathbb F}_2) \setminus \{e\} \} \quad\mbox{and} \quad \beta(n):= \min \{ \ell(w) \mid w \in {\mathbb F}_2^{(n)} \setminus \{e\} \}.$$ It is clear from \eqref{relation} that
\begin{equation} \label{rel}\alpha(2^{n}) \leq \beta(n).\end{equation}

We can think of $\alpha(n)$ resp.\ $\beta(n)$ as the girth the Cayley graph of the group ${\mathbb F}_2/\gamma_n({\mathbb F}_2)$ resp.\ ${\mathbb F}_2/\Gamma^{(n)}$ with respect to the image of the natural generating set of ${\mathbb F}_2$. It is clear that $\alpha(1)=\beta(0)=1$ and that $\alpha$ and $\beta$ are monotone increasing. 

Fox \cite[Lemma 4.2]{fox} showed $\alpha(n) \geq n/2$ and this was improved by Malestein-Putman to $\alpha(n) \geq n$ \cite[Theorem 1.2]{MR2737679}.
Since $[\gamma_n({\mathbb F}_2),\gamma_m({\mathbb F}_2))] \subset \gamma_{n+m}({\mathbb F}_2)$, we get
$$\alpha(n+m) \leq 2 \left(\alpha(n) + \alpha(m) \right).$$
Since in particular $\alpha(2n) \leq 4 \alpha(n)$, this suggests an asymptotic behaviour of the form $\alpha(n)  =O(n^2)$ for some constant $C>0$ and infinitely many $n \in \N$. This indeed was shown by Malestein-Putman \cite{MR2737679} (on an infinite subset of $\N$) and conjectured to be sharp. However, already the simple computation
\begin{equation} \label{basic}
\ell([[a,b],[b,a^{-1}]]) = \ell(aba^{-1}b^{-1}  ba^{-1}b^{-1}a   bab^{-1}a^{-1} a^{-1} bab^{-1}) \leq 8 \ell(a) + 6 \ell(b)
\end{equation}
and the observation $[[\gamma_n({\mathbb F}_2),\gamma_n({\mathbb F}_2)],[\gamma_n({\mathbb F}_2),\gamma_n({\mathbb F}_2)]] \subset \gamma_{4n}(\Gamma)$ suggests that it is enough to multiply the length by $14$ in order to increase the depth in the central series by a factor of $4$. So, this then suggests
$\alpha(n) =O(n^{\mu})$ for $\mu = \log_4(14) < 2.$ In what follows we want to make these considerations precise and try to minimize $\mu$. It remains to be an open question if $\mu = 1 + \varepsilon$ for all $\varepsilon>0$ is possible to achieve.

\begin{lemma} We have
$$\inf \left\{ \frac{\log_2(\alpha(n))}{\log_2(n)} \mid n \in \mathbb N \right\} =\lim_{n \to \infty} \frac{\log_2(\alpha(n))}{\log_2(n)}$$
and $$\inf \left\{\frac{\log_2(\beta(n))}{n} \mid n \in \mathbb N \right\}=\lim_{n \to \infty} \frac{\log_2(\beta(n))}{n}.$$
\end{lemma}
\begin{proof}
From the first inclusion in \eqref{relation}, we see that $\alpha(nm) \leq \alpha(n)\alpha(m).$ Indeed, let $w \in {\mathbb F}_2$ be the shortest non-trivial word in $\gamma_m(\mathbb F_2)$. Then, it is easy to see that $w$ and some cyclic rotation $w'$ of $w$ are free and of length $\alpha(m)$. Applying the shortest non-trivial word in $\gamma_n(\mathbb F_2)$ to $w$ and $w'$ yields some non-trivial element in $\gamma_{nm}(\mathbb F_2)$ of length less than or equal $\alpha(n)\alpha(m)$. Now, the first part of the lemma is implied by Fekete's Lemma. The second part follows in a similar way from Equation \eqref{clear}.
\end{proof}
In view of the preceding lemma, we set
$$\alpha := \lim_{n \to \infty} \frac{\log_2(\alpha(n))}{\log_2(n)} \quad\mbox{and} \quad \beta:= \lim_{n \to \infty} \frac{\log_2(\beta(n))}{n}.$$
By Fox' result \cite[Lemma 4.2]{fox} and inequality \eqref{rel}, we get
$1 \leq \alpha \leq \beta.$ Our main result is the following:
\begin{theorem}
\label{main}
Let $\mathbb F_2$ be the free group on two generators and $(\alpha(n))_{n \in \N}, (\beta(n))_{n \in \N}, \alpha$ and $\beta$ be defined as above.
\begin{enumerate}
\item We have
$$\alpha \leq \frac{\log_2(3+\sqrt{17}) - 1}{\log_2(1 + \sqrt{2})}= 1,4411...\ $$ or equivalently
$\alpha(n) \preceq n^{\frac{\log_2(3+\sqrt{17}) - 1}{\log_2(1 + \sqrt{2})} + \varepsilon}$ for all $\varepsilon>0$.
\item We have $$\log_2(3) \leq \beta \leq \log_2\left(3 + \sqrt{17} \right) -1= 1.8325...\ $$
or equivalently
$\log_2(3) \cdot n \leq \log_2(\beta(n)) \leq (\log_2\left(3 + \sqrt{17} \right) -1) \cdot n + o(n).$
\end{enumerate}
\end{theorem}

It is currently unclear to us how one could improve the upper bounds. Unfortunately, it seems even more unclear how to provide lower bounds for $\alpha$. The proof of the upper bounds follows from an explicit construction of short elements in the next section. The lower bound for $\beta$ is a consequence of Theorem \ref{lowerbound}, see Corollary \ref{corol}.

\section{The construction} 
\label{construction}
Recall that we consider ${\mathbb F}_2$ to be generated by letters $a$ and $b$.
We set $a_0:=a, b_0:=b$ and define recursively
$$a_{n+1} := [b_n^{-1},a_n], \quad b_{n+1}:=[a_n,b_n], \quad \forall n \in \N.$$ 

\begin{lemma} \label{nocancel}
For all $n \in \N$, the products $a_na_n, b_nb_n, a_n^{-1}b_n, b_n^{-1}a_n, a_nb_n^{-1}, b_na_n^{-1},a_n^{-1}b_n^{-1}$, and $ b_na_n$ involve no cancellation.
\end{lemma}
\begin{proof}
We prove the claim by induction, where the case $n=0$ is obvious. We check
$a_n^{-1}b_n = [b_{n-1}^{-1},a_{n-1}]^{-1}[a_{n-1},b_{n-1}] = [a_{n-1},b_{n-1}^{-1}][a_{n-1},b_{n-1}].$ The claim follows since $b_{n-1}a_{n-1}$ involves no cancellation. Similarly, $a_nb_n^{-1} = [b_{n-1}^{-1},a_{n-1}][b_{n-1},a_{n-1}]$ (and hence $b_na_n^{-1}$) involves no cancellation since $a_{n-1}^{-1}b_{n-1}$ has no cancellation; $a_n^{-1}b_{n}^{-1} = [a_{n-1},b_{n-1}^{-1}][b_{n-1},a_{n-1}]$ (and hence $b_na_n$) has no cancellation since $b_{n-1}b_{n-1}$ has no cancellation. Now, similarly $a_na_n = [b_{n-1}^{-1},a_{n-1}][b_{n-1}^{-1},a_{n-1}]$ has no cancellation since $a_{n-1}^{-1}b_{n-1}^{-1}$ has no cancellation, and finally $b_nb_n = [a_{n-1},b_{n-1}][a_{n-1},b_{n-1}]$ has no cancellation since $b_{n-1}^{-1}a_{n-1}$ has no cancellation. This proves the claim.
\end{proof}
\begin{lemma} We have $\ell(a_n)=\ell(b_n) \geq 2^n$ for all $n \in \N$.
\end{lemma}
\begin{proof} It follows from Lemma \ref{nocancel} that
\begin{eqnarray*}
\ell(b_n) &=& \ell(a_{n-1} b_{n-1} a_{n-1}^{-1} b_{n-1}^{-1}) \\
&=& \ell(a_{n-1} b_{n-1}) + \ell(b_{n-1}) + \ell(a_{n-1}) \\
&=& \ell(b_{n-1}^{-1}a_{n-1} b_{n-1} a_{n-1}^{-1})\\
&=& \ell(a_n).
\end{eqnarray*}
Now, it is obvious from this computation that $\ell(b_n) \geq 2 \ell(b_{n-1})$ for all $n \in \N$, and hence $\ell(b_n) \geq 2^n$ for all $n \in \N$.
This proves the claim. 
\end{proof}

\begin{lemma} \label{upper} For all $n\in \N$, we have
$\ell(b_n) \leq 3 \cdot \ell(b_{n-1}) + 2 \cdot \ell(b_{n-2}).$ In particular, there exists a constant $C'>0$, such that
$\ell(b_n) \leq C' \cdot \left(\frac{3+\sqrt{17}}{2} \right)^n$ for all $n \in \N$.
\end{lemma}
\begin{proof}
We estimate the length of $b_n$ in a straightforward way:
\begin{eqnarray*} 
\ell(b_n) &=& \ell([a_{n-1},b_{n-1}]) \\
&=& \ell([[b_{n-2}^{-1},a_{n-2}],[a_{n-2},b_{n-2}]]) \\
&\leq& \ell((b_{n-2}^{-1}a_{n-2}b_{n-2}a_{n-2}^{-1} a_{n-2}b_{n-2}a_{n-2}^{-1}b_{n-2}^{-1}))\\ 
&&+\ \ell([a_{n-2},b_{n-2}^{-1}]) + \ell([b_{n-2},a_{n-2}]) \\
&\leq& \ell(b_{n-2}^{-1}a_{n-2}b_{n-2}) + \ell(b_{n-2}) + \ell(a_{n-2}^{-1}) + \ell(b_{n-2}^{-1}) \\
&&+\ \ell([a_{n-2},b_{n-2}^{-1}]) + \ell([b_{n-2},a_{n-2}]) \\
&=&  3 \cdot \ell(b_{n-1}) + 2 \cdot \ell(b_{n-2}),
\end{eqnarray*}
where we used the equation $\ell(b_{n-2}^{-1}a_{n-2}b_{n-2}) = \ell(a_{n-1}) - \ell(a_{n-2})$ (a consequence of Lemma \ref{nocancel}) in the last equality. The estimate follows from the fact that $\frac{3 + \sqrt{17}}{2}$ is the largest root of the polynomial $p(\lambda)=\lambda^2 - 3 \lambda - 2$.
 This proves the claim.
\end{proof}

Our first result concerns the growth of the girth in the derived series.
\begin{proposition}
Let $\mu := \frac{3 + \sqrt{17}}{2} = 3,56155...\ .$ We have
$\beta(n)  \leq C' \cdot \mu^n$ for some constant $C'>0$ and infinitely many $n \in \N$. In particular, we get $\beta \leq \log_2(\mu)=1.8325...\ $.
\end{proposition}
\begin{proof}
We set $\delta(w) := \max\{ n \in \N \mid w \in \Gamma^{(n+1)}\}$. It is clear from the construction, that $\delta(b_n) \geq n$. Moreover, we clearly have $\beta(\delta(w)) \leq \ell(w)$. Thus, 
$$\beta(\delta(b_n)) = \ell(b_n) \leq C' \mu^{n} \leq C' \mu^{\delta(b_n)}.$$
This finishes the proof.
\end{proof}

Since $\alpha(2^n) \leq \beta(n)$, the previous result suggests
$\alpha(n) \leq C' n^{\log_2(\mu)}.$ We can improve the exponent by a factor $\log_2(1 + \sqrt{2})$.
Let $$\nu:= \frac{\log_2(3+\sqrt{17}) - 1}{\log_2(1 + \sqrt{2})}= 1,44115577304...\ .$$
\begin{proposition}
 We have $\alpha(n) \leq C' \cdot n^{\nu}$ for infinitely many $n \in \N$ and thus $\alpha \leq \nu$.
\end{proposition}
\begin{proof}
Note that we have the identities
\begin{equation} \label{eqrel}
[[a^{-1},b],[a,b]] = [[[a^{-1},b],a],[a,b]] \quad \mbox{and} \quad [[a^{-1},b],[b,a]] = [[[a^{-1},b],a],[b,a]].
\end{equation}
Indeed, we just check
$
[[a^{-1},b],a] = a^{-1}bab^{-1}a ba^{-1}b^{-1}aa^{-1}= [a^{-1},b] [a,b]$ and use that $[a,b]$ commutes with both $[a,b]$ and $[b,a]=[a,b]^{-1}$. This proves Equation \eqref{eqrel}.
We set 
\begin{equation}
\gamma(w):=\max\{n \mid w \in \gamma_n({\mathbb F}_2)\}. \quad \forall w \in \mathbb F_2.
\end{equation} Clearly, $\gamma(w_1w_2) \geq \min\{\gamma(w_1),\gamma(w_2)\}$ and $\gamma([w_1,w_2]) \geq \gamma(w_1) + \gamma(w_2)$. In order to proceed we need the following lemma.
\begin{lemma} \label{lower}
We have $\gamma(b_n) \geq 2 \gamma(b_{n-1}) + \gamma(b_{n-2})$ for all $n \in \N$. In particular, there exists a constant $C>0$, such that
$\gamma(b_n) \geq C \cdot (1+\sqrt{2})^n.$
\end{lemma}
\begin{proof}
We compute:
\begin{eqnarray*}
b_n &=& [a_{n-1},b_{n-1}] \\&=& [[b_{n-2}^{-1},a_{n-2}],[a_{n-2},b_{n-2}]] \\
&\stackrel{\eqref{eqrel}}{=}& [[[b_{n-2}^{-1}, a_{n-2}], b_{n-2}], [a_{n-2}, b_{n-2}]] \\
&=& [[a_{n-1}, b_{n-2}], b_{n-1}].
\end{eqnarray*}
This proves the claim since $\gamma(a_{n-1})= \gamma(b_{n-1})$ as $b_{n-2} a_{n-1} b_{n-2}^{-1} = b_{n-1}$.
The estimate on $\gamma(b_n)$ follows as before by a study of the growth of the recursively defined sequence $\gamma_n:= 2 \gamma_{n-1} + \gamma_{n-2}$.
\end{proof}

We are now ready to prove the upper bounds on $\alpha(n)$. Note that $\alpha(\gamma(b_n)) \leq \ell(b_n)$ for all $n \in \N$. Thus, as a consequence of Lemma \ref{lower} and Lemma \ref{upper}, we get
$$n \leq \frac{\log_2(\gamma(b_n)) - \log_2(C)}{\log_2(1+\sqrt{2})}$$ and hence
\begin{eqnarray*}\alpha(\gamma(b_n)) &\leq& \ell(b_n) \\ &\leq& C' \cdot \mu^{n} \\
&\leq& C' \exp\left(\frac{\log(\mu) \cdot (\log_2(\gamma(b_n)) - \log_2(C))}{\log_2(1+\sqrt{2})} \right) \\
&=& C' \exp\left(\frac{-\log(\mu)\log_2(C)}{\log_2(1 + \sqrt{2})} \right) \cdot (\gamma(b_n))^{\nu}.
\end{eqnarray*}
This proves the claim.
\end{proof}

\begin{question}
Can we prove better bounds of the form $\gamma(w) \leq \ell(w)^{\delta}$ for some $\delta<1$?
\end{question}

\section{Lower bounds for the derived series}

Again, we consider $\mathbb F_2$ -- the free group with generators $a,b$. For a subgroup $\Lambda \subset \mathbb F_2$, we define $${\rm girth}(\Lambda) := \min \{ \ell(w) \mid w \in \Lambda \setminus \{e\} \}.$$ 

\begin{theorem} \label{lowerbound}
Let $\Lambda \subset \mathbb F_2$ be a normal subgroup. Then, ${\rm girth}([\Lambda,\Lambda]) \geq 3 \cdot {\rm girth}(\Lambda)$
holds.
\end{theorem}
\begin{proof}
Recall, a subset $S \subset \mathbb F_2$ is called Nielsen reduced if \begin{enumerate}
\item[(i)] $u \neq e$, for all $u \in S$,
\item[(ii)] $\ell(uv) \geq \max\{\ell(u), \ell(v)\}$, for all $u,v \in S^{\pm1}$ with $uv \neq e$,
\item[(iii)] $\ell(uvw) > \ell(u) - \ell(v) + \ell(w)$, for all $u,v,w \in S^{\pm 1}$ with $uv \neq e$ and $vw \neq e$.
\end{enumerate}
It is well-known that $\Lambda$ has a Nielsen reduced basis \cite[Proposition 2.9]{lyndonschupp} -- let us denote it by $S$. We use the notation
$$|w| := \min \{ \ell(vwv^{-1}) \mid v \in \mathbb F_2 \}.$$
We will show that $|w| \geq 3 \cdot {\rm girth}(\Lambda)$ for all non-trivial $w \in [\Lambda,\Lambda]$.
Every element $w = [\Lambda,\Lambda]$ is a product of elements in $S^{\pm 1}$, so that that the exponent sum of each individual $s \in S$ is equal to zero. Hence, we may assume that $w = s w_1t w_2 s^{-1}w_3 t^{-1}$ or $w=sw_1s^{-1}w_2tw_3t^{-1}$ for some $w_1,w_2,w_3 \in \Lambda$ and $s,t \in S^{\pm}$ such that $st\neq e$ and $st^{-1} \neq e$. Since we are assuming that our basis for $\Lambda$ is Nielsen reduced, the cancellations from the left and right inside some element of $S$ cannot overlap and each will never touch more that one half of the word. Without loss of generality, we may assume that the cancellation that occurs in the product $t^{-1}\cdot s$ is the longest among the cancellations between all other letters that appear in $w$. Let us write $t=at_1$ and $s=as_1$ so that $t_1^{-1}s_1$ is reduced. Let us discuss the first case, i.e.\ $w=s w_1t w_2 s^{-1}w_3 t^{-1}$. Now, the cancellation in the product of $sw_1$ and $tw_2s^{-1}$ must be an initial segment $b$ of $a$, and similarily the cancellation  in the product of $tw_2s^{-1}$ and $w_3t^{-1}$ must be an initial segment $c$ of $a$. 
Since $\Lambda$ is a normal subgroup, we get that
$${\rm girth}(\Lambda) \leq \ell(tw_2s^{-1}) - 2 \ell(a), \quad {\rm girth}(\Lambda) \leq \ell(sw_1) - 2 \ell(b), \quad
{\rm girth}(\Lambda) \leq \ell(w_3t^{-1}) - 2 \ell(c).$$
Hence,
$$
3 \cdot {\rm girth}(\Lambda) \leq \ell(sw_1) + \ell(tw_2s^{-1}) + \ell(w_3t^{-1}) - 2 \ell(a) - 2\ell(b) - 2 \ell(c)= |w|.$$ 
In the second case, i.e.\ $w=sw_1s^{-1}w_2tw_3t^{-1}$, we consider the words $sws^{-1}, w_2$ and $tw_3t^{-1}$ and argue in a similar way. Indeed, the word $b$ cancelled in the product of $sws^{-1}$ and $w_2$ must be an initial segment of $a$. Similarly, the word $c$ cancelled in the product of $w_2$ and $tw_3t^{-1}$. Without loss of generality, $c$ is an initial segment of $b$. Now, we get
$${\rm girth}(\Lambda) \leq \ell(tw_2t^{-1}) - 2 \ell(a), \quad {\rm girth}(\Lambda) \leq \ell(sw_1s^{-1}) - 2 \ell(a), \quad
{\rm girth}(\Lambda) \leq \ell(w_2) - 2 \ell(c).$$ Hence, also in this case we get:
$$
3 \cdot {\rm girth}(\Lambda) \leq \ell(sw_1s^{-1}) + \ell(w_2) + \ell(tw_3t^{-1}) - 4 \ell(a)  - 2 \ell(c) \leq |w|.$$ 
This proves the claim.
\end{proof}

\begin{corollary} \label{corol}
We have ${\rm girth}(\mathbb F_2^{(n)}) \geq 3^{n}$. In particular, we get $\beta \geq \log_2(3) = 1.5849...$\ .
\end{corollary}

\section{Some applications}

\subsection{Nilpotent residually finiteness growth}
Following Khalid Bou-Rabee \cite{MR2851069} we define $F^{\rm nil}_{\mathbb F_2}(n)$ to be the smallest integer so that for every element $w \in \mathbb F_2$ of length less than or equal $n$, there exists a homomorphism to a finite nilpotent group of cardinality at most $F^{\rm nil}_{\mathbb F_2}(n)$ which does not map $w$ to the neutral element. Following \cite{MR2851069}, the growth behaviour determined by $F^{\rm nil}_{\mathbb F_2}$ is called the {\it nilpotent residual finiteness growth} of the free group.
Claim 1 in the proof of Theorem 3 in \cite{MR2851069} stated 
$$\exp(n^{1/2}) \preceq  F^{\rm nil}_{\mathbb F_2}(n).$$
Using the upper bound on $\alpha$ in Theorem \ref{main}, we can improve a little bit on this.

\begin{theorem} We have
$\exp(n^{\delta}) \preceq  F^{\rm nil}_{\mathbb F_2}(n)$ with
with $$\delta = \frac{\log_2(1 + \sqrt{2})}{\log_2(3+\sqrt{17}) - 1}= 0,69391... .$$
\end{theorem}
\begin{proof}
The proof is identical to the proof of Claim 1 on page 705 of \cite{MR2851069}. \end{proof}
\subsection{Almost laws for compact groups}

For every group $G$, an element $w \in \mathbb F_2$ gives rise to a natural word map $w \colon G \times G \to G$, which is just given by evaluation.
In \cite{MR3043070}, the second author proved that there exists a sequence of non-trivial elements $(w_n)_n$ in the free group on two generates, such that for every compact group $G$ and every neighborhood $V \subset G$ of the neutral element, there exists $m \in \N$ such that $w_n(G\times G) \subset V$ for all $n \geq m$. This statement is already non-trivial for a fixed compact group such as ${\rm SU}(2)$. Following \cite[Section 5.4]{breu}, we call such a sequence an {\it almost law} for the class of compact groups.

For a specific group like ${\rm SU}(k)$ with a natural metric, say $d(u,v)  := \|u-v\|$ where $\|.\|$ denotes the operator norm, it is natural to ask how long a word $w \in \mathbb F_2$ necessarily has to be, if we demand that $d(1_k,w(u,v)) < \varepsilon$ for all $u,v \in {\rm SU}(k)$. We set 
$$L_k(w) := \max \{ d(1_k,w(u,v)) \mid u,v \in {\rm SU}(k) \}.$$ In \cite[Remark 3.6]{MR3043070} it was claimed that there is a construction of an almost law $(w_n)_n$ as above such that for every $\varepsilon>0$, there exists a constant $C>0$ (which depends also on $k$) such that
$$L_k(w_n)  \leq \exp( -C \cdot \ell(w_n)^{\log_{14}4- \varepsilon} )$$
with $\log_{14}4= 0,5252... $. This construction relies on the basic idea that was already mentioned in connection with Equation \eqref{basic}. The more refined study in this paper yields:

\begin{theorem} \label{almost} Let $k \in \N$. There exists an almost law $(w_n)_n$ for ${\rm SU}(k)$ such that the following holds.
There exists a constant $C>0$ such that
$$L_k(w_n)  \leq \exp\left( -C \cdot \ell(w_n)^{\delta} \right)$$
with $$\delta = \frac{\log_2(1 + \sqrt{2})}{\log_2(3+\sqrt{17}) - 1}= 0,69391... .$$
\end{theorem}
\begin{proof}
Our basic method is a well-known contraction property of the commutator map in a Banach algebra. Let $k$ be fixed. In terms of the function $L_k$, Lemma 2.1.\ in \cite{MR3043070} says
\begin{equation} \label{zassenhaus}
L_k([w,v]) \leq 2 \cdot L_k(w)L_k(v).
\end{equation} 
We conclude from Corollary 3.3.\ in \cite{MR3043070} that there exist words $w,v \in \mathbb F_2$ which generate a free subgroup and satisfy $L_k(w),L_k(v) \leq \frac13$. Let us set $w_n:=a_n(w,v)$. It is clear that 
\begin{equation} \label{length}
\ell(w_n) \leq C'' \cdot \left(\frac{3 + \sqrt{17}}{2} \right)^n
\end{equation} for some constant $C''>0$. On the other side, Equation \eqref{zassenhaus} and the equation
$$b_n = [[a_{n-1},b_{n-2}],b_{n-1}]$$
from the proof of Lemma \ref{lower} shows that
$$L_k(w_n) \leq 4 \cdot L_k(w_{n-1})^2 L_k(w_{n-2})$$ or equivalently
$$- \log(2L_k(w_n)) \geq - 2\log(2L_k(w_{n-1})) - \log(2L_k(w_{n-2})).$$
Thus -- precisely as in the proof of Lemma \ref{lower} -- there exists a constant $D>0$ such that
\begin{equation} \label{small}
- \log (2L_k(w_n)) \geq D \cdot (1 + \sqrt{2})^n,
\end{equation}
for some constant $D>0$. Hence,
$$L_k(w_n) \stackrel{\eqref{small}}{\leq} \frac12 \exp\left(- D \cdot (1 + \sqrt{2})^n \right) \stackrel{\eqref{length}}{\leq} \exp\left(- C  \cdot \ell(w_n)^{\delta} \right)$$
for some constant $C$. This implies the claim.\end{proof}

It would be interesting to find a more direct relationship between the growth of the girth of the lower central series and the asymptotics encountered in Theorem \ref{almost}. It is presently unclear if $1 + \varepsilon$ for any $\varepsilon>0$ (or even for $\varepsilon=0$) is enough in Theorem \ref{almost}, see also Section 5.4 in \cite{breu} for a discussion of this question.

\section*{Acknowledgments}
We want to thank Jan-Christoph Schlage-Puchta and Dan Titus Salajan for interesting comments.

\end{document}